\documentclass{amsart}
\usepackage{amssymb,latexsym}

\usepackage{xcolor}

\newtheorem{thrm}{Theorem}[section]
\newtheorem{lem}[thrm]{Lemma}

\newtheorem{notation}[thrm]{Notation}

\theoremstyle{definition}
\newtheorem{definition}[thrm]{Definition}
\newtheorem{example}[thrm]{Example}

\theoremstyle{remark}

\numberwithin{equation}{section}

\def\Z{{\mathbb Z}}

\def\la{{\langle}}
\def\ra{{\rangle}}
\def\tl{{\unlhd}}
\def\fix{{\rm fix}}

\def\AGL{{\rm AGL}}
\def\Aut{{\rm Aut}}
\def\Stab{{\rm Stab}}

\def\cal{\mathcal}
\def\PStab{{\rm PStab}}

\def\WStab{{\rm WStab}}

\begin{document}
\title[]{Towards inductive proofs in algebraic combinatorics}

\author{Ted Dobson}
\address{University of Primorska, UP IAM, Muzejski trg 2, SI-6000 Koper, Slovenia and, University of Primorska, UP FAMNIT, Glagolja\v{s}ka 8, SI-6000 Koper, Slovenia}
\email{ted.dobson@upr.si}

\begin{abstract}
\smallskip
We introduce a new class of transitive permutation groups which properly contains the automorphism groups of vertex-transitive graphs and digraphs.  We then give a sufficient condition for a quotient of this family to remain in the family, showing that relatively straightforward induction arguments may possibly be used to solve problems in this family, and consequently for symmetry questions about vertex-transitive digraphs.  As an example of this, for $p$ an odd prime, we use induction to determine the Sylow $p$-subgroups of transitive groups of degree $p^n$ that contain a regular cyclic subgroup in this family.  This is enough information to determine the automorphism groups of circulant digraphs of order $p^n$.

\end{abstract}

\maketitle

For many problems on vertex-transitive graphs, a natural approach to take for the solution is induction.  For example, suppose we want to determine the automorphism group $\Aut(\Gamma)$ of a Cayley digraph $\Gamma$ of a fixed group $G$.  First, if the automorphism group is primitive, one can use the Classification of the Finite Simple Groups together with the O'Nan-Scott Theorem, and perhaps determine $\Aut(\Gamma)$.  In practice, there are many results in the literature that list all primitive groups that have certain orders or properties (see \cite{Li2003,LiS2003,LiebeckPS2010,LiebeckS1985a,LiebeckS1985}, for example).  In many cases, these known results are all that is needed to determine $\Aut(\Gamma)$.  Obviously, induction is not used in this analysis.  However, if $\Aut(\Gamma)$ is not primitive, then it has a nontrivial block system ${\cal B}$, and we may use the Embedding Theorem \cite[Theorem 4.3.1]{Book} of Kalu\v znin and Krasner to embed $\Aut(\Gamma)$ into a permutation group wreath product, say $G_1\wr G_2$.  We have thus reduced our larger problem of what is $\Aut(\Gamma)$ into three problems with smaller groups.  Namely, what are $G_1$ and $G_2$, and, how does one put them together to obtain $\Aut(\Gamma)$?  Ideally, the groups $G_1$ and $G_2$ would be the automorphism groups of smaller vertex-transitive graphs or digraphs.  However, examples are known where $G_1$ is not the automorphism group of a vertex-transitive graph or digraph \cite{Dobson2003b}, and examples are known where $G_2$ is not the automorphism group of a vertex-transitive graph or digraph \cite[Corollary 4.5.8]{Book}.  So while induction can theoretically get us started towards finding $\Aut(\Gamma)$, by itself it fails immediately as there is no general way of knowing if $G_1$ and $G_2$ are the automorphism groups of a graph or digraph.

Nonetheless, there are two inductive approaches that are used to study symmetry questions regarding vertex-transitive graphs and digraphs.  These are Cheryl Praeger's normal quotient lemma \cite{Praeger1985} and graph covers (see \cite[Section 5.5]{Book} and the references there).  These techniques are widely used by many authors, but have their own weaknesses.  The normal quotient lemma, for example, only applies to a small number of vertex-transitive graphs, namely those whose automorphism group is arc-transitive as well, and also does not apply to bipartite graphs.  Graph covers only apply to vertex-transitive graphs with a block system of a subgroup of the automorphism group with a proper intransitive semiregular subgroup (although we comment that Primoz Poto\v cnik and Micael Toledo \cite{PotocnikT2021} have recently removed this restriction provided more information is known).  Additionally, graph covers usually produce many different graphs whose quotient is the quotient one began with, and it may not be practical to determine which one of these graphs is the original graph.

In this paper, we introduce a new inductive technique for questions regarding the automorphism group of vertex-transitive digraphs.  We introduce a new family of permutation groups which can be conveniently placed in Wielandt's $k$-closed hierarchy (although not perfectly), and not only contains the automorphism group of all graphs and digraphs, but all $2$-closed groups.  For this new family, we give in Theorem \ref{quotient 5/2-closed} a sufficient condition for the quotient of a group in this class to be in the class.  This can be interpreted as a sufficient condition for induction, as outlined above, to at least have the potential to be used. To demonstrate that this can work in practice, we use this sufficient condition to inductively determine the Sylow $p$-subgroups of transitive groups in this family that contain a regular cyclic subgroup of odd prime-power order.  This will give the Sylow $p$-subgroups of transitive groups of circulant digraphs of odd prime-power order.  We remark that this is enough information to determine the automorphism groups of circulant digraphs of odd prime-power order, although for length reasons this will be done elsewhere \cite{Dobson2023epreprint}.  In other papers, we will begin to explore applications to the Cayley isomorphism problem \cite{Dobson2023bpreprint}, a generalization of the Polycirculant conjecture \cite{Dobson2023dpreprint}, and classifications of vertex-transitive digraphs of certain orders \cite{Dobson2023cpreprint}.

There is also an important philosophical idea behind this work.  The original idea behind this work is that if one can prove results regarding the automorphism group of some combinatorial object, then there must be some permutation group theoretic properties that the object imposes on its automorphism group, and rather than study the automorphism group itself, it may be useful to study permutation groups which have the group theoretic properties of the automorphism group.  That is what is done here, although the inspiration \cite[Lemma 2]{Dobson1995} for the class of groups considered in this paper had to be modified to obtain the sufficient condition given in Theorem \ref{quotient 5/2-closed}.  This also makes clear where algebraic tools are used and where combinatorial ideas are necessary.

We should also point out that our new class of permutation groups is broad enough to not only contain automorphism groups of vertex-transitive digraphs, but also automorphism groups of other combinatorial objects.  In fact, the automorphism group of a transitive ``directed and colored $1$-intersecting set system" (see Definition \ref{tuple def}) are also contained in our new class of groups (Theorem \ref{1-int 2.5 closed}), and combinatorial configurations are contained in the family of $1$-intersecting set systems.  Thus this work should also simplify proofs and provide new lines of attack for problems on these combinatorial objects as well.  In fact, it was recently shown \cite{MuzychukKP2001} that $\Z_n$ is a CI-group with respect to the class of symmetric configurations.  In \cite{Dobson2023bpreprint}, we will not only greatly generalize this result, but provide a shorter proof of the generalization.

\section{A refinement of the notion of $k$-closed groups}

We start with Wielandt's notion of $k$-closure.

\begin{definition}\label{karydefin}
Let $X$ be a set and $k\ge 1$ an integer.  A {\bf $k$-ary relational structure} on $X$ is an ordered pair $(X, U)$, where $U\subseteq X^k = \Pi_{i=1}^kX$. An {\bf automorphism} of a $k$-ary relational structure $X = (X,U)$ is a bijection $g\colon X\to X$ such that $g(U) = U$.  The set of all automorphisms of $X$ is a group, called the {\bf automorphism group} of $X$, and denoted $\Aut(X)$.

A group $G\le S_X$ is called {\bf $k$-closed} if $G$ is the intersection of the automorphism groups of some set of $k$-ary relational structures.  The {\bf $k$-closure of $G$}, denoted $G^{(k)}$, is the intersection of all $k$-closed subgroups of $S_X$ that contain $G$.
\end{definition}

Clearly $G^{(k)}$ is $k$-closed. There are other, alternative definitions of a $k$-closed group, and our next goal is to see one of these.

\begin{definition}
A {\bf color $k$-ary relational structure} is a $k$-ary relational structure $(X,U)$ in which each element of $U$ has been assigned a color.  An {\bf automorphism} of a color $k$-ary relational structure $W = (X,U)$ is a bijection $g\colon X\to X$ such that an element $u\in U$ is colored with color $i$ if and only if $g(u)\in U$ is colored with color $i$.  The {\bf automorphism group} of $W$, denoted $\Aut(W)$, is the group of all automorphisms of $W$.
\end{definition}

As we may distinguish different $k$-ary relations by coloring them with different colors, a group $G$ is $k$-closed if and only if it is the automorphism group of a color $k$-ary relational structure.

We now discuss our refinement of Wielandt's notion of $k$-closed groups for transitive permutation groups, and start with the property that the automorphism groups of digraphs and graphs has which we will use in the proof of our main result.  We will need some relatively standard terms in permutation group theory first.

\begin{definition}
Let $G\le S_n$ with orbit ${\cal O}$, and $g\in G$. Then $g$ induces a permutation on ${\cal O}$ by restricting the domain of $g$ to ${\cal O}$.  We denote the resulting permutation in $S_{\cal O}$ by $g^{\cal O}$.  The group $G^{\cal O} = \{g^{\cal O}:g\in G\}$ is the  {\bf transitive constituent} of $G$ on ${\cal O}$.
\end{definition}

\begin{definition}\label{imprimitive}
Let $X$ be a set and $G\le S_X$ be transitive.  A subset $B\subseteq X$ is a {\bf block} of $G$ if whenever $g\in G$, then $g(B)\cap B = \emptyset$ or $B$.  If $B = \{x\}$ for some $x\in X$ or $B = X$, then $B$ is a {\bf trivial block}.  Any other block is nontrivial.  Note that if $B$ is a block of $G$, then $g(B)$ is also a block of $B$ for every $g\in G$, and is called a {\bf conjugate block of $B$}.  The set of all blocks conjugate to $B$ is a partition of $X$, called a {\bf block system of $G$}.
\end{definition}

\begin{definition}
Let $X$ be a set and $G\le S_X$ be transitive.  If $N\tl G$, then the set of orbits of $N$ is a block system ${\cal B}$ of $G$.  We say ${\cal B}$ is a {\bf normal block system} of $G$.
\end{definition}

\begin{definition}
Let $G\le S_n$ be transitive with a block system ${\cal B}$.  By $\fix_G({\cal B})$ we denote the subgroup of $G$ that fixes each
block of ${\cal B}$ set-wise.  That is, $\fix_G({\cal B}) = \{g\in G:g(B) = B{\rm\ for\ all\ }B\in{\cal B}\}$.  If ${\cal C}$ is another block system of $G$ and each block of ${\cal B}$ is (properly) contained within a block of ${\cal C}$, we write (${\cal B}\prec{\cal C}$) ${\cal B}\preceq{\cal C}$, and say ${\cal B}$ {\bf refines} ${\cal C}$.  We denote the induced action of $G$ on the block system ${\cal B}$ by $G/{\cal B}$, and the action of an element $g\in G$ on ${\cal B}$ by $g/{\cal B}$.  That is, $g/{\cal B}(B) = B'$ if and only if $g(B) = B'$, and $G/{\cal B} = \{g/{\cal B}:g\in G\}$.  If ${\cal B}\preceq{\cal C}$, then $G/{\cal B}$ has a block system ${\cal C}/{\cal B}$ whose blocks are the set of blocks of ${\cal B}$ whose union are blocks of ${\cal C}$.  In this case, we denote the set of blocks of $B\in{\cal B}$ contained in $C$ by $C/{\cal B}$.  Hence ${\cal C}/{\cal B} = \{C/{\cal B}:C\in{\cal C}\}$.
\end{definition}

It is easy to see that $\fix_G({\cal B})$ is a normal subgroup of $G$.  We are now ready to prove the preliminary results which will be needed for the definition of the new class of permutation groups.

\begin{lem}
Let $G\le S_n$ be transitive with normal block system ${\cal B}$, and $B\in{\cal B}$.  Let $K_1,K_2\le\fix_G({\cal B})$ such that $K_i^B = 1$ and $K_i^{B'}$ is trivial or transitive for every $B'\in{\cal B}$, $i = 1,2$.  Then $\la K_1,K_2\ra^B = 1$, and $\la K_1,K_2\ra^{B'}$ is trivial or transitive for every $B'\in{\cal B}$.  Consequently, there exists a maximal subgroup $K(B)$ of $\fix_G({\cal B})$ such that $K(B)^B = 1$ and $K(B)^{B'}$ is either trivial or transitive for every $B'\in{\cal B}$.
\end{lem}

\begin{proof}
As $K_1,K_2\le\fix_G({\cal B})$, $\la K_1,K_2\ra\le\fix_G({\cal B})$.  As every element of $K_1$ and $K_2$ is the identity on $B$, $\la K_1,K_2\ra^B = 1$.  Let $B'\in{\cal B}$.  If $K_1^{B'} = 1$ and $K_2^{B'} = 1$, then analogous to $B$, $\la K_1,K_2\ra^{B'} = 1$.  Otherwise, at least one of $K_1^{B'}$ and $K_2^{B'}$ is transitive, in which case  $\la K_1,K_2\ra^{B'}$ is transitive.  It is now clear that $K(B)$ as in the statement exists.  It is simply the subgroup generated by all subgroups $K\le\fix_G({\cal B})$ such that $K^B = 1$ and $K^{B'} = 1$ or is transitive for every $B\in{\cal B}$.
\end{proof}

\begin{definition}
We call the subgroup $K(B)$ of $\fix_G({\cal B})$ the {\bf wreath} stabilizer of $B\in{\cal B}$, and denote it by $\WStab_G(B)$.
\end{definition}

The name ``wreath" stabilizer is chosen because if $G = \Aut(\Gamma)$ with $\Gamma$ a vertex-transitive graph, then the wreath stabilizer will give for which blocks $B'\in{\cal B}$ we must have that every element of $B$ is adjacent to every element of $B'$ or no element of $B$ is adjacent to every element of $B'$.  This situation is reminiscent of when $\Gamma$ can be written as a nontrivial wreath product.  Our next result shows that the wreath stabilizer does act like a stabilizer.

\begin{lem}\label{WStabs conjugate}
Let $G\le S_n$ be transitive with normal block system ${\cal B}$.  Let $B, B'\in {\cal B}$, and $g\in G$ such that $g(B) = B'$.  Then $\WStab_G(B') = g\WStab_G(B)g^{-1}$.
\end{lem}

\begin{proof}
As $g\WStab_G(B)g^{-1}$ fixes every element of $B'$, we see $g\WStab_G(B)g^{-1}\le \WStab_G(B')$.  Similarly, $g^{-1}\WStab_G(B')g\le \WStab_G(B)$.  Thus $\WStab_G(B')\le g\WStab_G(B)g^{-1}$, and so $\WStab_G(B) = g\WStab_G(B')g^{-1}$.
\end{proof}

\begin{definition}\label{automorphismtooldef}
Let $G$ be a transitive group that has a normal block system ${\cal B}$.  Define a relation $\equiv$ on ${\cal B}$ by $B\equiv B'$ if and only if $\WStab_G(B) = \WStab_G(B')$.
\end{definition}

Note that $B\not\equiv B'$ means that $\WStab_G(B)^{B'}$ is transitive.

\begin{lem}\label{equiv properties}
The relation $\equiv$ is an equivalence relation on ${\cal B}$, the set ${\cal E}$ of the unions of the equivalence classes of $\equiv$ is a block system of $G$, and ${\cal B}$ is a refinement of ${\cal E}$.
\end{lem}

\begin{proof}
That $\equiv$ is reflexive is trivial. Let $B,B'\in{\cal B}$.  As $\WStab_G(B) = \WStab_G(B')$ if and only if $\WStab_G(B') = \WStab_G(B)$,  $\equiv$ is symmetric.  Verifying the transitive property is equally difficult.  So $\equiv$ is an equivalence relation.  To see the set ${\cal E}$ of the unions of the equivalence classes of $\equiv$ is a block system, by \cite[Theorem 2.3.2]{Book}, we need only observe that $\equiv$ is a $G$-congruence.  Let $B,B'\in{\cal B}$, and $g\in G$. By Lemma \ref{WStabs conjugate} $\WStab_G(B) = \WStab_G(B')$ if and only if $$\WStab_G(g(B)) = g\WStab_G(B)g^{-1} = g\WStab_G(B')g^{-1} = \WStab_G(g(B')).$$ Hence $B\equiv B'$ if and only if $g(B)\equiv g(B')$ and $\equiv$ is a $G$-congruence.  Finally, it is clear ${\cal B}$ refines ${\cal E}$ as each block of ${\cal E}$ is a union of blocks of ${\cal B}$.
\end{proof}

\begin{definition}
We call $\equiv$ the {\bf ${\cal B}$-restricting equivalence relation of $G$}, and ${\cal E}$ the {\bf ${\cal B}$-fixer block system of $G$}.
\end{definition}

\begin{notation}
Let $g\in S_n$ and $X\subseteq\Z_n$ such that $g(X) = X$.  By $g\vert_X$ we mean the element of $S_n$ such that $g\vert_X(y) = g(y)$ if $y\in X$, while $g\vert_X(y) = y$ if $x\notin X$.
\end{notation}

We are now ready to define our new class of permutation groups.

\begin{definition}
Let $G\le S_n$ be transitive.  For $H\le G$  transitive with normal block system ${\cal B}_H$ let ${\cal E}_{H,{\cal B}}$ be the ${\cal B}_H$-fixer block system of $H$. Suppose that for every transitive subgroup $H\le G$, every normal block system ${\cal B}_H$ of $H$, and every $g\in G$ that fixes each block of ${\cal B}_H$ contained in $E\in{\cal E}_{H,{\cal B}}$ setwise, we have $g\vert_E\in G$.  We will then say that $G$ is {\bf $5/2$-closed}. For a group $G$, the {\bf $5/2$-closure of $G$}, denoted $G^{(5/2)}$, is the intersection of all $5/2$-closed groups which contain $G$.
\end{definition}

\begin{example}
A transitive group $G\le S_n$ for which every transitive subgroup has no nontrivial normal block systems is necessarily $5/2$-closed.  In particular, every group of prime degree $p$ is $5/2$-closed.  As $\AGL(1,p)$ is $2$-transitive, $(\AGL(1,p))^{(2)} = S_p$.  Thus if $p\ge 5$, $(\AGL(1,p))^{(2)}\not = (\AGL(1,p))^{(5/2)}$, and the family of $2$-closed groups is different from the family of $5/2$-closed groups.
\end{example}

The next example shows that the set of all $5/2$-closed groups is not contained in the set of all $3$-closed groups.

\begin{example}
We saw in the previous example that the group $M_{23}$ in its representation on $23$ points is $5/2$-closed.  However, $M_{23}$ is $4$-transitive and so $3$-transitive as well.  Then $(M_{23})^{(3)} = S_{23}$, so $M_{23}$ is not $3$-closed.  The same argument gives the same result for any $3$-transitive group of prime degree.
\end{example}

The next example shows that there are $3$-closed groups which are not $5/2$-closed.  This, together with the previous example, gives that the class of $3$-closed groups and $5/2$-closed groups are not comparable as classes.

\begin{example}
Let $P$ be a Sylow $p$-subgroup of $N_{S_{p^2}}((\Z_{p^2})_L)$, where $p$ is an odd prime.  Then $P$ has order $p^3$, and the stabilizer of the two points $0$ and $1$ is trivial.  By \cite[Theorem 5.12]{Wielandt1969} we see that $P$ is $3$-closed.  Also, $P$ has a unique normal block system ${\cal B}$ that is the set of orbits of the subgroup generated by $x\mapsto x + p$.  But $\WStab_P(0) = \la x\mapsto (1 + p)x\ra = \la\delta\ra$, and  $\WStab_P(0)$ is transitive on every block of ${\cal B}$ that does not contain $0$ as the only fixed points of $\la\delta\ra$ are multiples of $p$.  Then $\delta\vert_{B'}\in P^{(5/2)}$ for every $B'\in{\cal B}$, and $\fix_{P^{(5/2)}}({\cal B})$ has order $p^p$. So $\vert P\vert = p^{p + 1}$ and is not $P$ as $p > 2$.
\end{example}

The following result is obvious after observing that every transitive subgroup of $H$ is a transitive subgroup of $G$.

\begin{lem}
Let $H\le G$ be transitive groups.  Then $H^{(5/2)}\le G^{(5/2)}$.
\end{lem}

In the literature, there is an equivalence relation that is more commonly encountered than $\equiv$ given in Definition \ref{automorphismtooldef}. It is the equivalence relation $\sim$ defined on ${\cal B}$ by $B\sim B'$ if and only if $\PStab_{\fix_G({\cal B})}(B) = \PStab_{\fix_G({\cal B})}(B')$, where $\PStab_{\fix_G({\cal B})}(B)$ is the pointwise stablizer of $B$ in $\fix_G({\cal B})$.  The next result and its following example give the relationship between these two equivalence classes.

\begin{lem}\label{WStab vs PStab}
Let $G\le S_n$ be transitive with normal block system ${\cal B}$.  Then $\WStab_G(B)\tl\fix_G({\cal B})$ for all $B\in{\cal B}$.  Thus $\WStab_G(B)\tl\PStab_{\fix_G({\cal B})}(B)$, and the set of equivalence classes of $\equiv$ refines the set of equivalence classes of $\sim$.
\end{lem}

\begin{proof}
Clearly if $g\in\fix_G({\cal B})$, then $g$ fixes $B\in{\cal B}$ and $(g^{-1}\WStab_G(B)g)^{B'}$ is trivial or transitive for every $B'\in{\cal B}$ if and only if $\WStab_G(B)^{B'}$ is trivial or transitive for every $B'\in{\cal B}$.  Thus $g^{-1}\WStab_G(B)g\le \WStab_G(B)$ and $\WStab_G(B)\tl \fix_G({\cal B})$.  Hence $\WStab_G(B)\tl\PStab_{\fix_G({\cal B})}(B)\le\fix_G({\cal B})$.  Also, if $B\sim B'$, then $\WStab_G(B) = \WStab_G(B')$, so $B\equiv B'$.  Thus the equivalence class of $\equiv$ that contains $B$ also contains $B'$, and so contains the equivalence class of $\sim$ that contains $B$.
\end{proof}

\begin{definition}
A permutation group $G\le S_n$ is {\bf semiregular} if $\Stab_G(x) = 1$ for every $x\in \Z_n$, and $G$ is {\bf regular} if $G$ is both semiregular and transitive.
\end{definition}

\begin{definition}
Let $G$ be a group and $g\in G$.  Define $g_L\colon G\to G$ by $g_L(x) = gx$.  The {\bf left regular representation of $G$}, denoted $G_L$, is $G_L = \{g_L:g\in G\}$.  That is, $G_L$ is the group of all left translations of $G$.
\end{definition}

It is straightforward to verify that $G_L$ is a regular group isomorphic to $G$.  The next example shows that $\sim$ and $\equiv$ are different equivalence relations.

\begin{example}
Let $p$ be an odd prime, and define $\delta,\tau_2,\tau_3:\Z_p^3\to\Z_p^3$ by $\delta(i,j,k) = (i,j,k + i)$, $\tau_2(i,j,k) = (i,j + 1,k)$, and $\tau_3(i,j,k) = (i,j,k + 1)$.  Note $\delta$ normalizes $(\Z_p^3)_L$.  The group $G = \la \delta,(\Z_p^3)_L\ra$ has order $p^4$, and its subgroup $H = \la\tau_2,\tau_3,\delta\ra$ is normal of index $p$ in $G$.  Then the set of orbits of $H$ is a normal block system ${\cal B}$ with blocks of size $p^2$.  Let $B_0\in{\cal B}$ contain $(0,0,0)$.  Then $\PStab_{\fix_G({\cal B})}(B_0) = \la \delta\ra$.  Also, $B_0$ is the only block of ${\cal B}$ fixed point-wise by $\delta$, and so the equivalence class of $\sim$ is simply the block system ${\cal B}$ as each block is only related under $\sim$ to itself.  As $\la\delta\ra^B$ is semiregular of order $p$ (and so not transitive) for any $B_0\not = B\in{\cal B}$, we see that $\WStab_G(B_0) = 1$.  Then $\WStab_G(B) = 1$ for all $B\in{\cal B}$, and $B\equiv B'$ for every $B,B'\in{\cal B}$.  We conclude that $\sim$ and $\equiv$ are different equivalence relations.
\end{example}

\begin{definition}
A transitive permutation group $G$ is {\bf quasiprimitive} if every nontrivial normal subgroup of $G$ is transitive.
\end{definition}

Sometimes though, the equivalence relations $\sim$ and $\equiv$ are the same.

\begin{lem}\label{equivalence relation exists}
Let $G\le S_n$ be transitive with a normal block system ${\cal B}$.  If $\fix_G({\cal B})^B$ is quasiprimitive, then $\PStab_{\fix_G({\cal B})}(B) = \WStab_G(B)$ for every $B\in{\cal B}$.  Consequently, the equivalence relations $\equiv$ and $\sim$ on ${\cal B}$ are the same.
\end{lem}

\begin{proof}
We have already seen in Lemma \ref{WStab vs PStab} that $\WStab_G(B)\le \PStab_{\fix_G({\cal B})}(B)$ for every $B\in{\cal B}$.  It is easy to show that $\PStab_{\fix_G({\cal B})}(B)\tl\fix_G({\cal B})$ for every $B\in{\cal B}$.  Finally, as $\fix_G({\cal B})^B$ is quasiprimitive, we see $\PStab_{\fix_G({\cal B})}(B)^{B'}$ is trivial or transitive for every $B'\in{\cal B}$.  Hence $\PStab_{\fix_G({\cal B})}(B)\le \WStab_G(B)$ for every $B\in{\cal B}$ and $\PStab_{\fix_G({\cal B})}(B) = \WStab_G(B)$ for every $B\in{\cal B}$.
\end{proof}

Some words about the choice of calling this family $5/2$-closed.   With Wielandt's hierarchy of groups, a $1$-closed transitive group is simply a symmetric group \cite[Theorem 5.11]{Wielandt1969}.  As Wielandt showed \cite[Theorem 5.10]{Wielandt1969} that if $\ell\ge k$ then $(G^{(k)})^{(\ell)} = G^{(k)}$, every $k$-closed group is $(k + 1)$-closed.  So as $k$ grows, the size of the class of $k$-closed group gets larger.  We will see in Theorem \ref{1-int 2.5 closed} that every $2$-closed group is $5/2$-closed.  Thus, in placing this new class of groups within Wielandt's hierarchy, it should be to the ``right" of $2$-closed, but not to the right of $3$-closed.  We then arbitrarily picked the midpoint between $2$ and $3$.

Finally, we end this section with a discussion of an obvious question.  Are there analogues of $5/2$-closed groups for values of $k\ge 3$?  The short answer is definitely a ``Yes".  Indeed, the motivation for our definition of $5/2$-closed groups comes from a result on graphs \cite[Lemma 2]{Dobson1995}, and analogues of this result exist for $k$-closed groups for all $k\ge 3$ \cite[Lemma 12]{Dobson2010b}.  For fixed $k\ge 3$, an analogous definition of $\equiv$ on ${\cal B}$ would be that $B\equiv_k B'$ if and only if whenever $B_1,\ldots,B_{k-1}\in{\cal B}$ are distinct and $\WStab_G(B_i) = \WStab_G(B_j)$, $1\le i\le k-1$, then $\WStab_G(B) = \WStab_G(B_1)$ if and only if $\WStab_G(B') =\WStab_G(B_1)$.  It is straightforward to verify that $\equiv_k$ is an equivalence relation on ${\cal B}$, and is a $G$-congruence, as wreath stabilizers are conjugate.  We could then define $(2k+1)/2$-closed groups as we defined $5/2$-closed groups by substituting the equivalence classes of $\equiv_k$ for $\equiv$.  However, it is not clear that this definition would be useful.  Additionally, it is not clear what should be done for block systems that do not have at least $k$ blocks.  In particular, in \cite[Lemma 12]{Dobson2003} an analogue of \cite[Lemma 2]{Dobson1995} is proven for ternary relational structures which have $2$ blocks of prime size $p$.  So it is not clear that $\equiv_k$ by itself will give the appropriate analogue of \cite[Lemma 2]{Dobson1995}.  Consequently, we will not formally define a $(2k + 1)/2$-closed group at this time.

\section{Quotients of $5/2$-closed groups}

In this section, we will show that under appropriate circumstances the quotient of a $5/2$-closed group is $5/2$-closed.  We begin with a basic result about blocks that will be needed.

\begin{lem}\label{quotient normal}
Let $G$ be a transitive group with a normal block system ${\cal B}$.  Suppose $G/{\cal B}$ has a (normal) block system ${\cal C}/{\cal B}$.  Let ${\cal C}$ be the set of unions of the blocks of ${\cal B}$ contained in a block of ${\cal C}/{\cal B}$.  Then ${\cal C}$ is a (normal) block system of $G$.
\end{lem}

\begin{proof}
Let $C/{\cal B}\in {\cal C}/{\cal B}$, and $C$ be the union of the set of all blocks of ${\cal B}$ contained in $C/{\cal B}$.  Let $g\in G$ such that $g(C)\cap C\not = 1$.  Then $g$ maps some block of ${\cal B}$ in $C$ to another block of ${\cal B}$ in $C$, and so $g/{\cal B}(C/{\cal B}) = C/{\cal B}$.  As $C/{\cal B}$ is a block of $G/{\cal B}$, $g$ maps the set of blocks of ${\cal B}$ in $C$ to the set of blocks of ${\cal B}$ in $C$.  Hence $g(C) = C$ and $C$ is a block of $G$.  Thus ${\cal C}$ is a block system of $G$.

For the ``normal" part, let $N \le G$ such that $N/{\cal B} = \fix_{G/{\cal B}}({\cal C}/{\cal B})$.  Let $C\in{\cal C}$.  Then $N$ is transitive on the set of blocks of ${\cal B}$ contained in $C$.  As $\fix_G({\cal B})/{\cal B} = 1$, $\fix_G({\cal B})\le N$, and so $N$ is transitive on $C$.  Thus ${\cal C}$ is the set of orbits of $N$, and ${\cal C}$ is a normal block system of $G$.
\end{proof}

Thus, in the quotient, (normal) block systems of the quotient correspond to (normal) block systems of the group.  We call the block system ${\cal C}$ of $G$, as in the preceding lemma, the block system of $G$ {\bf induced} by ${\cal C}/{\cal B}$.  Our main result of this section follows easily from the next result.  The preceding result will be used implicitly.

\begin{lem}\label{key tool}
Let $G$ be a $5/2$-closed transitive group with a normal block system ${\cal B}$, with ${\cal E}$ the ${\cal B}$-fixer block system of $G$.  Suppose there is $H\le G$ that is transitive for which $H/{\cal B}$ has a nontrivial normal block system ${\cal C}/{\cal B}$, with ${\cal E}'/{\cal B}$ the ${\cal C}/{\cal B}$-fixer block system of $H/{\cal B}$.  Let $H'\le G$ be maximal such that $H'/{\cal B} = H/{\cal B}$.  If ${\cal E}\preceq{\cal E}'$, then the ${\cal C}$-fixer block system of $H'$ is ${\cal E}'$.
\end{lem}

\begin{proof}
By Lemma \ref{quotient normal}, ${\cal C}$ is a normal block system of $H$.  As ${\cal B}$ is a block system of $G$, ${\cal B}$ is a block system of both $H$ and $H'$, and as $H'/{\cal B} = H/{\cal B}$, ${\cal C}/{\cal B}$ is a block system of $H'/{\cal B}$.  Hence ${\cal C}$ is a block system of $H'$.  Similarly, ${\cal E}'$ is a block system of $H'$.  
Let $K\le\fix_{H'}({\cal C})$ such that $(K/{\cal B})^{C/{\cal B}}$ is trivial or transitive for every $C\in{\cal C}$.  We will first show that there exists a subgroup $K'\le \fix_{H'}({\cal C})$ such that $K'/{\cal B} = K/{\cal B}$ and if $(K/{\cal B})^{C/{\cal B}} = 1$, then $(K')^C = 1$ and otherwise that $(K')^C = K^C$.

Let $C\in{\cal C}$ such that $(K/{\cal B})^{C/{\cal B}} = 1$.  Let $k\in K$.  Then $k(B) = B$ for every $B\in{\cal B}$ contained in $C$.  Let $E'\in{\cal E}'$ that contains some $B\in{\cal B}$ contained in $C$.  Let $C'\in{\cal C}$ with $C'\subseteq E'$.  As $\WStab_{G/{\cal B}}(C/{\cal B}) = \WStab_{G/{\cal B}}(C'/{\cal B})$, we see that $k$ fixes setwise each block of ${\cal B}$ contained in $E'$.  Let $D$ be the union of all $E'\in{\cal E}'$ that contain a block of ${\cal B}$ contained in $C$. Then $k$ fixes each block of ${\cal B}$ contained in $D$ setwise.  As $D$ is a union of blocks of ${\cal E}'$ and ${\cal E}\preceq{\cal E}'$, $D$ is a union of blocks of ${\cal E}$.  As $G$ is $5/2$-closed, we see $k\vert_E\in G$ for every $E\in{\cal E}$ with $E\subseteq D$.  Hence $k\vert_D\in G$.

Now, $\hat{k} = [k(k\vert_D)^{-1}]$ has the property that $\hat{k}/{\cal B} = k/{\cal B}$ and $\hat{k}^D = 1$.  Additionally, $\hat{k}^C = k^C$ for every block of ${\cal C}$ not contained in $D$.  Repeating this procedure for every block $C'$ of ${\cal C}$ for which $(K/{\cal B})^{C'/{\cal B}} = 1$, we see there is $k'\in\fix_{H'}({\cal C})$ such that $k'/{\cal B} = k/{\cal B}$, $(k')^C = 1$ if $k$ fixes setwise each block of ${\cal B}$ contained in $C$, and $(k')^C = k^C$ for every other block $C\in{\cal C}$.  Hence $K' = \la k':k\in K\ra$ has the properties that  $K'/{\cal B} = K/{\cal B}$ and if $K/{\cal B}$ fixes every block of ${\cal B}$ contained in $C\in{\cal C}$, then $(K')^C = 1$, and otherwise $(K')^C = K^C$.  Note this implies that $(K')^C$ is trivial for every block $C\in{\cal C}$ such that $(K/{\cal B})^{C/{\cal B}}$ is trivial.

Now fix $C\in{\cal C}$ for which $(K')^C\not = 1$.  Then $(K')^C$ is transitive on the blocks of ${\cal B}$ contained in $C$.  Let $E\in{\cal E}$ with $C\subseteq E$.  As $G$ is $5/2$-closed we have that $\fix_{H'}({\cal B})\vert_E\le G$.  The group $\la K',\fix_{H'}({\cal B})\vert_E\ra$ is then transitive on $C$.  Also, as ${\cal E}\preceq{\cal E}'$, we see that $\la K',\fix_G({\cal B})\vert_E\ra$ is trivial on every block of ${\cal C}$ that $K'$ is trivial on.  Replacing $K'$ by $\la K',\fix_{H'}({\cal B})\vert_E\ra$ and repeating this procedure for every $C\in{\cal C}$ for which $(K')^C\not = 1$, we may assume without loss of generality that $K'$ is transitive on every block of ${\cal C}$ that $K'$ is not trivial on.

We now see that $\WStab_{H'}(C)^{C'}$ is trivial if and only if $\WStab_{H/{\cal B}}(C/{\cal B})^{C'/{\cal B}}$ is trivial and $\WStab_{H'}(C)^{C'}$ is transitive if and only if $\WStab_{H/{\cal B}}(C/{\cal B})^{C'/{\cal B}}$ is transitive, and so $C\equiv C'$ if and only if $C/{\cal B}\equiv C'/{\cal B}$.  Thus the ${\cal C}$-fixer block system of $H'$ is ${\cal E}'$.
\end{proof}

\begin{thrm}\label{quotient 5/2-closed}
Let $G$ be a $5/2$-closed transitive group with a normal block system ${\cal B}$ with ${\cal E}$ the ${\cal B}$-fixer block system of $G$.  Suppose that whenever ${\cal B}\preceq{\cal C}$ is a normal block system of any transitive subgroup $H$ of $G$ then the block system induced by the ${\cal C}/{\cal B}$-fixer block system ${\cal E}_{\cal C}/{\cal B}$ of $H/{\cal B}$ is refined by ${\cal E}$.  Then $G/{\cal B}$ is $5/2$-closed.
\end{thrm}

\begin{proof}
Let $H\le G$ be transitive.  Then ${\cal B}$ is a block system of $H$.  Let ${\cal B}\preceq{\cal C}$ be such that ${\cal C}/{\cal B}$ is a normal block system of $H/{\cal B}$.  Let $H'\le G$ be the largest subgroup of $G$ such that $H'/{\cal B} = H/{\cal B}$, with ${\cal E}'/{\cal B}$ the ${\cal C}/{\cal B}$-fixer block system of $H'/{\cal B}$.  As $H'/{\cal B} = H/{\cal B}$, we see that ${\cal E}'/{\cal B}$ is the ${\cal C}/{\cal B}$-fixer block system of $H/{\cal B}$.  By hypothesis, ${\cal E}\preceq{\cal E'}$.  By Lemma \ref{key tool}, the ${\cal C}$-fixer block system of $H'$ is ${\cal E}'$.  Let $g/{\cal B}\in G/{\cal B}$ such that $g/{\cal B}$ fixes each block of ${\cal C}/{\cal B}$ contained in $E'/{\cal B}\in{\cal E}'/{\cal B}$ set-wise.  Then $g\in G$ fixes each block of ${\cal C}$ contained in $E'$ setwise, and as $G$ is $5/2$-closed, we see that $g\vert_{E'}\in G$.  Then $g/{\cal B}\vert_{E'/{\cal B}}\in G/{\cal B}$.  As $H$ and ${\cal C}$ are arbitrary, $G/{\cal B}$ is $5/2$-closed by definition.
\end{proof}

\section{Combinatorial objects with $5/2$-closed automorphism groups}

In this section we explore the natural question of which combinatorial objects have automorphism groups that are $5/2$-closed? While we do not have a characterization of such combinatorial objects, we first give a large class of combinatorial objects whose automorphism groups are $5/2$-closed which contains graphs, digraphs, and configurations.  

\begin{definition}\label{tuple def}
Let $X$ be a finite set and $T$ a subset of $\cup_{i=1}^n X^i$ for some $n < \vert X\vert$.  Hence $T$ is a set of $k_i$-tuples with elements in $X$ for some set of integers $k_1,\ldots,k_r < n$.  We will call $T$ a {\bf tuple system} on $X$.  Let $S = \{\{u_1,\ldots,u_t\}:(u_1,\ldots,u_t)\in T\}$, so $S$ is the set of coordinates of each element of $T$, and we call it the  {\bf set system corresponding to $T$}.  Let $m\ge 0$ be an integer.  We say a set system $S$ is $m$-intersecting if for every $U_1,U_2\in S$ with $U_1\not = U_2$, $\vert U_1\cap U_2\vert\le m$.  A tuple system $T$ will also be called $1$-intersecting if its corresponding set system is $1$-intersecting.
\end{definition}

\begin{example}
The arc set of a finite digraph $\Gamma$ is a tuple system, with corresponding set system the set of edges of the underlying simple graph.  Also, observe that $A(\Gamma)$ is $1$-intersecting, as two different arcs can only share at most one head or tail.
\end{example}

\begin{example}
A {\bf configuration} is an ordered triple $(P,L,{\cal I})$, where $P$ is a set whose elements are called {\bf points}, $L$ is a set whose elements are called {\bf lines}, and ${\cal I}$ is an {\bf incidence relation} (or subset of $P\times L$).  If $(p,\ell)\in{\cal I}$, we say the point $p$ is on, or incident with, the line $\ell$ or the line $\ell$ contains, or is incident with, the point $p$.  The incidence relation of a configuration satisfies several properties, one of which is that two points are on at most one line and two lines contain at most one point in common.  So the set of lines of a configuration is a $1$-intersecting set system regardless of the other properties (which can be found in \cite[pg. 15]{Grunbaum2009}).  The point here is that $1$-intersecting set systems properly contain the set of configurations.  We remark that in design theory, a configuration is often referred to as a $1$-design.
\end{example}

\begin{definition}
Let $T$ be a tuple (set) system.  We say $T$ is {\bf colored} if each tuple (set) in $T$ has been identified with a color (or integer).  Hence there is a function $f:T\to {\mathbb N}$ and $f(t)$ is simply the color of the tuple (set) $t\in T$.
\end{definition}

We may think of an uncolored tuple (set) system $T$ as a colored tuple (set) system simply by assigning the same color to every element of $T$.

\begin{definition}
Let $T$ and $T'$ be colored tuple (set) systems defined on $X$ and $Y$, respectively.  An {\bf isomorphism} from $T$ to $T'$ is a bijection $f:X\to Y$ that preserves tuples (sets) and colors.  That is, $f(t)\in T'$ if and only if $t\in T$ and for every color $i$, the image of the set of all tuples (sets) in $T$ colored $i$ is the set of all tuples (sets) colored $i'$ in $T'$, for some color $i'$.  An {\bf automorphism} of $T$ is an isomorphism from $T$ to $T$ that fixes each color, and the set of all automorphisms of $T$ is the {\bf automorphism group} of $T$, denoted $\Aut(T)$.  A colored tuple (set) system $T$ is {\bf point-transitive} if $\Aut(T)$ is transitive on $X$.
\end{definition}

\begin{thrm}\label{1-int 2.5 closed}
Let $T$ be a point-transitive colored $1$-intersecting tuple system on $X$.  Then $\Aut(T)$ is $5/2$-closed.
\end{thrm}

\begin{proof}
We may assume without loss of generality that $\Aut(T)$ has some transitive subgroup $G$ with a normal block system ${\cal B}$ for which the ${\cal B}$-fixer block system ${\cal E}$ of $H$ is not $\{X\}$ as otherwise $\Aut(T)$ is trivially $5/2$-closed.  Thus ${\cal E}$ has at least two blocks.  Let $t = (t_1,\ldots,t_k)\in T$ be of color $c$, and $s = \{t_i:1\le i\le k\}$.

Suppose that for some $1\le i < j\le k$ we have $t_i,t_j\in E\in{\cal E}$, $t_i\not = t_j$, and $t_\ell\not\in E$ for some $1\le \ell\le k$.  That is, that there are at least two different elements of $t$ in $E$ and at least one element of $t$ not in $E$.  Let $B'\in{\cal B}$ such that $t_\ell\in B'$, and $E'\in{\cal E}$ such that $B'\subseteq E'$.   As $t_\ell\not\in E$, $B\not\equiv B'$ for every $B\in {\cal B}$ with $B\subseteq E$.  Hence $\WStab_G(B)^{B'}$ is transitive, and there exists $\gamma\in \fix_G({\cal B})$ such that $\gamma^B = 1$ and $\gamma(t_\ell)\not = t_\ell$ for every $B\in{\cal B}$ with $B\subseteq E$.  Then $\gamma(t)\not = t$ as $t_\ell$ is a coordinate of $t$, and $\gamma(t)$ and $t$ have color $c$.  Then $\vert\gamma(s)\cap s\vert \ge 2$ as $t_i\not = t_j$ are in $s$, a contradiction.  So every coordinate of $t$ is contained in $E$.

Suppose $\vert s\cap E\vert \le 1$ for every $E\in{\cal E}$.  We are not assuming the coordinates of $t$ are all distinct.  Let $s = \{t_{i_1},t_{i_2},\ldots,t_{i_j}\}$ (so $j$ is the number of distinct elements of $s$, and $j\le k$), and for each $1\le \ell\le j$, let $T_{i_\ell}$ be the set of coordinates of $t$ that are $t_{i_\ell}$.  Then $\sum_{\ell=1}^j\vert T_{i_\ell}\vert = k$. Also, there exist distinct blocks $E_1,\ldots,E_j$ of ${\cal E}$ with $t_{i_\ell}\in E_\ell$, $1\le \ell\le j$.  There also exist blocks $B_1,\ldots,B_j\in{\cal B}$ such that $t_{i_\ell}\in B_\ell$, and $B_1,\ldots,B_j$ are all distinct as $B_\ell\subseteq E_\ell$, $1\le \ell\le j$.    Now, for each $1\le \ell,m\le j$ with $\ell\not= m$, we have $E_\ell\not = E_m$ and so $\WStab_G(B_\ell)^{B_m}$ is transitive. This implies that $(t_1',\ldots,t_k')\in T$, where $t_i' = t_i$ if $i\in T_\ell$ and for any $1\le m\le j$, $m\not = \ell$, we have that for all $i\in T_m$, $t_i'$ are equal but an arbitrary element of $B_m$.  Successively choosing $\ell = 1$ and $m = 2,\ldots,j$, we see that $(t_1,y_2,\ldots,y_k)\in T$ where $y_i\in B_m$ is arbitrary but equal for all  $i\in T_m$, $2\le m\le j$, and $y_i = t_i$ if $t_i\in E_1$, and are all colored $c$.  As $\fix_G({\cal B})^{B_1}$ is transitive, we have that $(y_1,y_2,\ldots,y_k)\in T$ for every $y_i\in B_\ell$ is arbitrary but equal for all $i\in T_\ell$, $1\le i\le k$ and $1\le i\le j$, and are all colored $c$.

Now let $g\in\Aut(T)$ such that $g$ fixes set-wise all blocks of ${\cal B}$ contained in a block $E\in{\cal E}$.  If $\vert s\cap E\vert\ge 2$ for some $E\in{\cal E}$, then every coordinate of $t$ is contained in $E$.  Hence $g\vert_{E}(t)\in T$ and is colored $c$.  If $\vert s\cap E\vert = 0$, then $g\vert_E(t) = t\in T$.  Otherwise, we will use the notation of the previous paragraph.  As $(y_1,y_2,\ldots,y_k)\in T$ for every $y_i\in B_\ell$ is arbitrary but equal for all $i\in T_\ell$, $1\le i\le k$ and $1\le\ell\le j$, and are all colored $c$, we see that $g\vert_E(t_1,\ldots,t_k)\in T$ and is colored $c$.  Thus $G$ is $5/2$-closed.
\end{proof}

\begin{example}\label{set to tuple}
Let $S$ be a $1$-intersecting set system.  Just as we identify the edge set of a graph $\Gamma$ with the arc set $\{(x,y):xy\in E(\Gamma)\}$ to obtain a digraph, we may identify $S$ with a tuple system $T$.  Namely, we let $T$ be the colored tuple system that is the union of all $r!$ $r$-tuples whose set of coordinates are an element $s$ of $S$ of length $r$.  We color each tuple of $T$ the same color as its corresponding set $s$.  Note that $\Aut(S) = \Aut(T)$.  Conversely, if $T$ is a colored tuple system in which for every $r$-tuple with underlying set $s$ we have all $r!$ $r$-tuples on $s$ and all are colored the same color, then we may identify $T$ with a colored set system in the obvious manner, and $\Aut(T) = \Aut(S)$ here as well.  In this way, we may identify a set system, and in particular, a configuration, with a $1$-intersecting tuple system.
\end{example}

Hence if one proves a group theoretic result that is true for all transitive $5/2$-closed permutation groups, it also holds for all $2$-closed groups, which in our language here is the automorphism group of a colored tuple system, as well as the automorphism groups of configurations!

\section{The $5/2$-closure of a wreath product is the wreath product of the $5/2$-closures}

In this section, we will show that some commonly used properties of $2$-closed groups are shared by $5/2$-closed groups.  In particular, in Theorem \ref{wreath 5/2 closed} we will show that the $5/2$-closure of the wreath product of two permutation groups is $5/2$-closed.  The same result for $2$-closed groups is given in \cite[Theorem 5.1]{Cameronetal2002}.  As a consequence, we show in Lemma \ref{5/2 closure of p-group} that the $5/2$-closure of a $p$-group is a $p$-group, an analogue of \cite[Exercise 5.28]{Wielandt1969} needed for our main application.  We begin with some basic definitions.

\begin{definition}
Let $G\le S_X$ and $H\le S_Y$.  Define the {\bf wreath product of $G$ and $H$}, denoted $G\wr H$, to be the set of all permutations of $X\times Y$ of the form $(x,y)\mapsto (g(x),h_x(y))$, where $g\in G$ and each $h_x\in H$.
\end{definition}

It is straightforward to show that $G\wr H$ is a permutation group.

\begin{definition}
Let $G\le S_X$ and $H\le S_Y$ both be transitive.  Let $B_x = \{x\}\times Y$, $x\in X$, and ${\cal B} = \{B_x:x\in X\}$.  Then ${\cal B}$ is a block system of $G\wr H$ by \cite[Theorem 4.3.7]{Book}.  It is called the {\bf lexi-partition of $G\wr H$ corresponding to $Y$}.
\end{definition}

\begin{lem}\label{wreath block}
Let $G_1\le S_m$ and $G_2\le S_n$ be transitive.  If ${\cal B}$ is a normal block system of $G_2$, then ${\cal C} = \{\{i\}\times B:i\in\Z_m, B\in{\cal B}\}$ is a normal block system of $G_1\wr G_2$.
\end{lem}

\begin{proof}
Let ${\cal D}$ be 
the lexi-partition of $G_1\wr G_2$ with respect to $\Z_n$ (where we view $G_2$ as permutating $\Z_n$).  Then ${\cal C}\preceq{\cal D}$.  Let $D\in{\cal D}$ and $C\in{\cal C}$ such that $C\subseteq D$.  Observe that $C$ is a block of $\Stab_{G_1\wr G_2}(D)^D\cong G_2$, so $C$ is block of $G_1\wr G_2$ by \cite[Theorem 2.3.10]{Book}.  Hence ${\cal C}$ is a block system of $G_1\wr G_2$. Then $\fix_{G_1\wr G_2}({\cal C}) = \{(i,j)\mapsto (i,g_i(j)):g_i\in \fix_{G_2}({\cal B})\}$.  We see the set of orbits of $\fix_{G_1\wr G_2}({\cal C})$ is ${\cal C}$, and as ${\cal C}$ is a block system of $G_1\wr G_2$, $\fix_{G_1\wr G_2}({\cal C})\tl G_1\wr G_2$.  Thus ${\cal C}$ is a normal block system of $G_1\wr G_2$.
\end{proof}

\begin{thrm}\label{wreath 5/2 closed}
Let $G_1\le S_m$ and $G_2\le S_n$ be transitive.  Then $(G_1\wr G_2)^{(5/2)} = G_1^{(5/2)}\wr G_2^{(5/2)}$.
\end{thrm}

\begin{proof}
We will first show that $G_1^{(5/2)}\wr G_2^{(5/2)}$ is $5/2$-closed.  As $G_1\wr G_2\le G_1^{(5/2)}\wr G_2^{(5/2)}$, this will give that $(G_1\wr G_2)^{(5/2)}\le G_1^{(5/2)}\wr G_2^{(5/2)}$.  So let ${\cal B}$ be a normal block system of some transitive subgroup $H$ of $G_1^{(5/2)}\wr G_2^{(5/2)}$.  Let ${\cal E}$ be the ${\cal B}$-fixer block system of $H$, and $\gamma\in G_1\wr G_2$ fix $E\in{\cal E}$ and every block of ${\cal B}$ contained in $E$.  As ${\cal B}$ and $\gamma$ are arbitrary, by definition it suffices to show that $\gamma\vert_E\in G_1^{(5/2)}\wr G_2^{(5/2)}$.  Let ${\cal C}$ be the lexi-partition of $G_1^{(5/2)}\wr G_2^{(5/2)}$ with respect to $\Z_n$ (we consider $S_n$ as permutating the elements of $\Z_n$).  We may assume without loss of generality that $H$ is maximal with ${\cal B}$ a block system of $H$.  By \cite[Lemma 5]{DobsonM2015a} (or \cite[Theorem 4.3.7]{Book}), we see that ${\cal B}\preceq{\cal C}$ or ${\cal C}\prec {\cal B}$.  We will consider the two cases separately.  Set $G = G_1^{(5/2)}\wr G_2^{(5/2)}$.

Suppose first that ${\cal B}\preceq{\cal C}$.  Let $B\in{\cal B}$, $C\in{\cal C}$, and $E\in{\cal E}$ with $B\subseteq C$ and $B\subseteq E$.  As ${\cal C}$ is a block system of $G$, ${\cal C}$ is also a block system of $H$.  By the Embedding Theorem \cite[Theorem 4.3.1]{Book}, we see there is $L\le G$ that is permutation isomorphic to $(H/{\cal C})\wr(\Stab_H(C)^C)$.  Let ${\cal B}_C = \{B\in{\cal B}:B\subseteq C\}$.  As ${\cal B}$ is a normal block system of $H$, ${\cal B}_C$ is a normal block system of $\Stab_H(C)^C$.  By Lemma \ref{wreath block}, ${\cal B}$ is a normal block system of $L$.  By the assumption that $H$ is maximal with ${\cal B}$ a block system, we see $L\le H$.  So ${\cal E}\preceq{\cal C}$.  As $\gamma$ fixes $B$, $\gamma$ fixes $C$, and thus $\gamma^C\in G_2^{(5/2)}$.  Recall $\Stab_H(C)^C\le G_2$ and ${\cal B}_C$ is a normal block system of $\Stab_H(C)^C$.  Additionally, ${\cal E}_C = \{E\in {\cal E}:E\subseteq C\}$ is refined by the ${\cal B}_C$-fixer block system of $\Stab_H(C)^C$.  So $(\gamma^C)\vert_E\in G_2^{(5/2)}$.  Hence $\gamma\vert_E\in G_1^{(5/2)}\wr G_2^{(5/2)}$.

Suppose ${\cal C}\prec{\cal B}$, in which case ${\cal B}/{\cal C}$ is a normal block system of $H/{\cal B}$.  Let ${\cal E}_1$ be the ${\cal B}/{\cal C}$-fixer block system of $H/{\cal C}$.  As $h\vert_C\le G$ for every $h\in\Stab_H(C)$ and $C\in{\cal C}$, ${\cal E}/{\cal C} = {\cal E}_1$.  Then $(\gamma/{\cal C})\vert_{E/{\cal C}}\in G_1^{(5/2)}$.  As there exists $\hat{\gamma}\in G$ with $\hat{\gamma}(i,j) = ([(\gamma/{\cal C})\vert_{E/{\cal C}}]^{-1}(i),j)$, we see $\hat{\gamma}\gamma\in\fix_G({\cal C}) = 1_{S_m}\wr G_2^{(5/2)}$.  We conclude that $\gamma\vert_E\in G$.  Hence $G_1^{(5/2)}\wr G_2^{(5/2)}$ is $5/2$-closed, and so $(G_1\wr G_2)^{(5/2)}\le G_1^{(5/2)}\wr G_2^{(5/2)}$.

To see $G_1^{(5/2)}\wr G_2^{(5/2)}\le (G_1\wr G_2)^{(5/2)}$, let $H\le G_2$ be transitive and ${\cal B}$ a normal block system of $G_2$.  Then $G_1\wr H\le G_1\wr G_2$ and ${\cal D} = \{\{i\}\times B:i\in\Z_m, B\in{\cal B}\}$ is a normal block system of $G_1\wr G_2$ by Lemma \ref{wreath block}.  For $B\in{\cal B}$, let $D_B = \{0\}\times B$.  Fix $B\in{\cal B}$. Consider $\WStab_{G_1\wr H}(D_B)$.  Note that for any $D\in{\cal D}$ of the form $D = \{i\}\times B'$ for any $i\not = 0$ and $B\not = B'\in{\cal B}$, we have $\WStab_{G_1\wr H}(D_B)^D$ is transitive.  Also, $\WStab_{G_1\wr H}(D_B)^{D_{B'}} = 1$ if and only if $\WStab_H(B)^{B'} = 1$.  We conclude that if ${\cal E}$ is the ${\cal B}$-fixer block system of $G_2$, then the ${\cal D}$-fixer block system of $G_1\wr H$ is ${\cal F} = \{\{i\}\times E:i\in\Z_n, E\in{\cal E}\}$.

Now $\gamma\in G_2$ fixes $E\in{\cal E}$ and every block of ${\cal B}$ contained in $E$ if and only if there is $\hat{\gamma}\in G_1\wr H$ that fixes $\{0\}\times E$ and fixes every block of ${\cal D}$ contained in $\{0\}\times E$ ($\hat{\gamma}$ can be chosen to be, for example, the map which on $\{0\}\times E$ is $(0,e)\mapsto (0,\gamma(e))$ and is trivial everywhere else).  Additionally, $\gamma\vert_E\in G_2$ if and only if $\gamma\vert_{\{0\}\times E}\in G_1\wr G_2$.  We conclude that $G_1\wr G_2^{(5/2)}\le (G_1\wr G_2)^{(5/2)}$.

Now let $H\le G_1$ be transitive and ${\cal B}$ a normal block system of $H$. Then $H\wr G_2\le G_1\wr G_2$ and ${\cal D} = \{B\times \Z_n:B\in{\cal B}\}$ is a normal block system of $H\wr G_2$.  For $B\in{\cal B}$, let $D_B = B\times\Z_n$.  Fix $B\in{\cal B}$. Consider $\WStab_{H\wr G_2}(D_B)$.  Note $\WStab_{H\wr G_2}(D_B)^{D_{B'}} = 1$ if and only if $\WStab_H(B)^{B'} = 1$ and $\WStab_{H\wr G_2}(D_B)^{D_{B'}}$ is transitive if and only if $\WStab_H(B)^{B'}$ is transitive.  We conclude that if ${\cal E}$ is the ${\cal B}$-fixer block system of $H$, then the ${\cal D}$-fixer block system of $H\wr G_2$ is ${\cal F} = \{E\times\Z_n:E\in{\cal E}\}$.

Let ${\cal C}$ be the lexi-partition of $G_1\wr G_2$ with respect to $\Z_n$ (so ${\cal C} = \{\{i\}\times\Z_n:i\in\Z_m\})$.  Now, $\gamma\in H$ fixes $E\in{\cal E}$ and every block of ${\cal B}$ contained in $E$ if and only if there is $\tilde{\gamma}\in H\wr G_2$ such that $\tilde{\gamma}/{\cal C} = \gamma$.  Then $\gamma\vert_E\in G_1^{(5/2)}$ if and only if $\tilde{\gamma}\vert_{E\times \Z_n}\in G_1^{(5/2)}\wr G_2$.  We conclude that $G_1^{(5/2)}\wr G_2\le (G_1\wr G_2)^{(5/2)}$, and so $G_1^{(5/2)}\wr G_2^{(5/2)}\le (G_1\wr G_2)^{(5/2)}$.  The result follows.
\end{proof}

\begin{definition}
Let $G\le S_m$ be a group, and $n$ a positive integer.  By $G\wr_nG$, we mean the group $G\wr G\wr\cdots\wr G$, where the wreath product is performed $n$ times.
\end{definition}

The next result generalizes \cite[Exercise 5.28]{Wielandt1969}.

\begin{lem}\label{5/2 closure of p-group}
Let $p$ be prime, $n$ a positive integer, and $P\le S_{p^n}$ be a transitive $p$-group.  Then $P^{(5/2)}$ is a $p$-group.
\end{lem}

\begin{proof}
If $n = 1$, then $P\cong(\Z_p)_L$ and has no nontrival normal block systems.  Hence $P^{(5/2)} = P$.  So we assume $n\ge 2$.  As $P$ is a transitive $p$-group, it is contained in a Sylow $p$-subgroup of $S_{p^n}$ which by \cite[\S 2]{Passman1968} is isomorphic to $\Z_p\wr_n\Z_p$.  So we assume $P\le \Z_p\wr_n\Z_p$.  It is clear from the definition of a $5/2$-closed group, that if $H\le G$ are transitive groups, then $H^{(5/2)}\le G^{(5/2)}$.  Inductively applying Theorem \ref{wreath 5/2 closed} we see that $(\Z_p\wr_n\Z_p)^{(5/2)}$ is $5/2$-closed.  Then $P^{(5/2)}\le (\Z_p\wr_n\Z_p)^{(5/2)} = \Z_p\wr_n\Z_p$ and $P^{(5/2)}$ is a $p$-group.
\end{proof}

\section{$5/2$-closed groups of odd prime-power order that contain a regular cyclic subgroup}

In this section, we will give an application of the work done already and determine the Sylow $p$-subgroups of $5/2$-closed permutation groups of degree $p^n$ that contain a regular cyclic subgroup, where $p$ is an odd prime and $n\ge 1$.  These turn out to be exactly the same as the Sylow $p$-subgroups of $2$-closed permutation groups of degree $p^n$ that contain a regular cyclic subgroup \cite{Dobson2023epreprint}.  While the Sylow $p$-subgroups of $2$-closed permutation groups of degree $p^n$ that contain a regular cyclic subgroup are not in the literature, they have been known for some time.  There are two independent proofs of this result.  Chronologically, the first is due to Klin and P\"oschel \cite{KlinP1981}, using the method of Schur.  The second, using permutation group theoretic techniques, can be found in \cite{Dobson1995}.  This application is chosen for this paper for two reasons.  First, it is well past time that a proof of the classification of automorphism groups of circulant digraphs of order $p^n$ appears in the literature (and this will not be completed until \cite{Dobson2023epreprint}).  Second, and from the point of view of this paper most importantly, is that this is the clearest and easiest application of the inductive techniques developed in this paper that the author is aware of.

We begin with a final general result, and then restrict our discussion to transitive permuation groups of degree $p^n$ that contain a regular cyclic subgroup.

\begin{lem}\label{E F relationship}
Let $G\le S_n$ be transitive with normal block systems ${\cal B}\prec{\cal C}$.  Let ${\cal E}$ be the ${\cal B}$-fixer block system of $G$ and ${\cal F}/{\cal B}$ the ${\cal C}/{\cal B}$-fixer block system of $G/{\cal B}$.  Then ${\cal C}\preceq{\cal E}$ or ${\cal E}\prec{\cal F}$.
\end{lem}

\begin{proof}
We assume that ${\cal C}$ does not refine ${\cal E}$.  Let $C\in {\cal C}$ and $E\in{\cal E}$ such that $C\cap E\not = \emptyset$.  Our assumption ensures that $E$ does not contain every block of ${\cal B}$ contained in $C$.  Now, if $E$ intersects two different blocks of ${\cal F}$, then there is a block $C'\in{\cal C}$ with $C'\not\subseteq E$ such that the induced action of $\WStab_{G/{\cal B}}(C'/{\cal B})$ on $C/{\cal B}$ is transitive.  Then $E$ contains every block of ${\cal B}$ contained in $C$, a contradiction.  Hence $E$ intersects only one block of ${\cal F}$, and ${\cal E}\preceq{\cal F}$.  Note that if ${\cal E} = {\cal F}$, then ${\cal C}\preceq{\cal F} = {\cal E}$, so ${\cal C}\preceq{\cal E}$ or ${\cal E}\prec{\cal F}$.
\end{proof}

Define $\tau_n:\Z_{p^n}\to\Z_{p^n}$ by $\tau(i) = i + 1\ ({\rm mod\ } p^n)$.  Thus $\tau_n$ generates $(\Z_{p^n})_L$. As $\la\tau_n\ra$ is cyclic of order $p^n$, $\la\tau_n\ra$ contains a unique normal subgroup of order $p^i$, for each $0\le i\le n$, namely $\la\tau_n^{p^{n-i}}\ra$.  Thus, for $0\le i\le n$, the group $\la\tau_n\ra$ has a (possibly trivial) block system ${\cal B}_i$ of $p^{n-i}$ blocks of size $p^{i}$, that is the set of orbits of $\la\tau_n^{p^{n-i}}\ra$.  Note that these block systems are the only block systems of $\la\tau_n\ra$, and hence if $\la\tau_n\ra\le G\le S_{p^n}$ and ${\cal B}$ is a block system of $G$, then ${\cal B} = {\cal B}_i$ for some $0\le i\le n$.  Furthermore, the blocks of ${\cal B}_i$ are cosets of the unique subgroup of $\Z_{p^n}$ of order $p^{i}$.   Thus if $B_{i,j} = j + p^{n-i}\Z_{p^{i}}$, where $j\in\Z_{p^{n-i}}$, then  ${\cal B}_i = \{B_{i,j}:j\in\Z_{p^{n-i}}\}$.  Also observe that $${\cal B}_0\prec{\cal B}_{1}\prec\ldots\prec{\cal B}_{n-1}\prec{\cal B}_n.$$

The next result is a special case of \cite[Lemma 23]{Dobson2008a}.

\begin{lem}\label{bottom}
Let $p$ be an odd prime, $n\ge 2$, and $P\le S_{p^n}$ be a transitive $p$-group that contains $\tau_n$.  If $P\not = \tau_n$, then $\vert\fix_P({\cal B})\vert\ge p^2$.
\end{lem}

\begin{lem}\label{p^kfixers}
Let $p$ be an odd prime, $n\ge 2$, and $P\le S_{p^n}$ be a transitive $p$-group that contains $\tau_n$.  Let ${\cal E}_2/{\cal B}_1$ be the ${\cal B}_2/{\cal B}_1$-fixer block system of $P$.  Then ${\cal E}_1\preceq{\cal E_2}$.
\end{lem}

\begin{proof}
By Lemma \ref{E F relationship}, we see that either ${\cal B}_2\preceq{\cal E}_1$ or ${\cal E_1}\prec{\cal E}_2$.  We may assume that ${\cal E}_1\not = {\cal B}_2$ as otherwise the result is trivially true as ${\cal B}_2\preceq{\cal E}_2$.  Suppose ${\cal B}_2\prec{\cal E}_1$.  As $p$ is odd, $\Z_p\wr\Z_p$ has order $p^{p+1} > p^3$.  Let $\gamma\in\fix_P({\cal B}_2)$, and $B_2\in{\cal B}_2$.  Set $\delta = \gamma^{-1}\tau_n^{p^{n-2}}\gamma\tau_n^{-p^{n-2}}$.  As $\fix_P({\cal B}_2)/{\cal B}_1$ is elementary abelian, we see that $\delta\in\fix_P({\cal B}_1)$.  As ${\cal B}_2\prec{\cal E}_1$, we see that $\delta^{B_2}\in \la\tau_n^{p^{n-1}}\ra^{B_2}$.  Hence $\gamma^{B_2}$ normalizes $\la\tau_n^{p^{n-2}}\ra^{B_2}$.  As $N_{S_{p^2}}((\Z_{p^2})_L)$ has a Sylow $p$-subgroup of order $p^3$, we see that $\fix_P({\cal B}_2)^{B_2}\not\cong\Z_p\wr\Z_p$.  As $\fix_P({\cal B}_2)$ contains the regular cyclic $\la\tau_n^{p^{n-2}}\ra^{B_2}$, it follows by \cite[Lemma 4]{DobsonW2002} that $\fix_P({\cal B}_2)^{B_2}$ does not contain a regular elementary abelian subgroup.  Thus if $\gamma^{B_2}$ is a $p$-cycle on the $p$ blocks of ${\cal B}_1$ contained in $B_2$, then it must be the case that $\gamma^{B_2}$ is a $p^2$-cycle as otherwise $\la \gamma,\tau_n^{p^{n-1}}\ra^{B_2}$ is a regular subgroup isomorphic to $\Z_p\times\Z_p$.  Hence $\gamma^p\in\fix_P({\cal B}_1)$, and $(\gamma^p)^{B_2} = 1$ if and only if $(\gamma/{\cal B}_1)^{B_2/{\cal B}_1} = 1$.  Thus ${\cal E}_1\preceq{\cal E}_2$.
\end{proof}

\begin{lem}\label{p^kfixercor}
Let $p$ be an odd prime, $n\ge 2$, and $P\le S_{p^n}$ be a transitive group that contains $\tau_n$.  Let ${\cal E}_i/{\cal B}_{i-1}$ be the ${\cal B}_i/{\cal B}_{i -1}$-fixer block system of $P/{\cal B}_{i - 1}$, $1\le i\le n$.  Then ${\cal E}_{i-1}\preceq{\cal E_i}$ for every $1\le i\le n$.
\end{lem}

\begin{proof}
We proceed by induction on $n$.  If $n = 2$, then ${\cal E}_2 = \{\Z_{p^2}\}$ and so ${\cal E}_1\preceq{\cal E}_2$.  Assume the result holds for all $n-1\ge 2$ and let $P\le S_{p^n}$ be a transitive group that contains $\tau_n$.  Then ${\cal B}_1$ is a normal block system of $P$, and ${\cal E}_2\preceq{\cal E}_3\preceq\ldots\preceq{\cal E}_n$ by the induction hypothesis.  By Lemma \ref{p^kfixers} ${\cal E}_1\preceq{\cal E}_2$ and the result follows.
\end{proof}

\begin{definition}
Let $p$ be prime and $n\ge 1$ an integer.  Define a {\bf primary key space} ${\bf K}_{p^n}$ to be the set of all integer vectors $(k_1,\ldots,k_n)$ satisfying the following two properties:
\begin{enumerate}
\item  $k_i < i$ for each $1\le i\le n$, and
\item $k_{i-1}\le k_i$ for each $2\le i\le n$.
\end{enumerate}
\noindent A vector in ${\bf K}_{p^n}$ is called a {\bf primary key}.
\end{definition}

Note that the primary key space ${\bf K}_{p^n}$ is independent of the choice of $p$, and that as $k_i < i$, we have $k_1 = 0$.  Primary keys were used in \cite{Muzychuk2004} to define partitions of $\Z_{p^n}$.  While we will not prove this here, the partition produced by a primary key when $p$ is odd is the set of orbits of the point stabilizer of a Sylow $p$-subgroup of a $2$-closed transitive group which contains $(\Z_{p^n})_L$.  We will use primary keys to construct the Sylow $p$-subgroups of $5/2$-closed transitive groups which contain $(\Z_{p^n})_L$.

\begin{definition}
Let ${\bf k} = (k_1,k_2,\ldots,k_n)$ be a primary key.  For fixed $k_i$, we define a subgroup $P_{i,k_i}$ by $P_{i,k_i} = \la \tau_n^{p^{i-1}}\vert_{B_{n-k_i}}:B_{n-k_i}\in{\cal B}_{n-k_i}\ra$.  We then define a group $\Pi({\bf k}) = \la P_{i,k_i}:1\le i\le n\ra$, and call it the subgroup {\bf corresponding} to ${\bf k}$.
\end{definition}

As $k_1 = 0$, the group $P_{1,0} = \la\tau_n^{p^0}\vert_{B_n}\ra = \la\tau_n\ra$, where $B_n = \Z_{p^n}$.  As the set of orbits of $\la\tau_n^{p^{i-1}}\ra$ is the unique block system of $\la\tau_n\ra$ with blocks of size $p^{n-(i-1)} = p^{n-i+1}$, we see the set of orbits of $\la\tau_n^{p^{i-1}}\ra$ is ${\cal B}_{n-i+1}$.  As $k_i \le i-1$, we see that $n-k_i \ge n-i+1$, and so ${\cal B}_{n-i+1}\preceq{\cal B}_{n-k_i}$.  That is, the orbits of $\la\tau_n^{p^{i-1}}\ra$ are contained within blocks of ${\cal B}_{n-k_i}$.  Thus $\tau_n^{p^{i-1}}\vert_{B_{n-k_i}}$ is well-defined for every $B_{n-k_i}\in{\cal B}_{n-k_i}$, and so $P_{i,k_i}$ is well-defined for every $1\le i\le n$.

One more observation about primary keys.  As one reads ${\cal B}_0\prec{\cal B}_1\prec\ldots\prec{\cal B}_n$ from left to right, the size of the blocks in a block system get larger.  As one reads a primary key from left to right, the size of the blocks in ${\cal B}_{n-k_i}$ get smaller or stay the same.  In order to cause as little confusion in the literature as possible, we are following Muzychuk's definition of a primary key \cite[Section 2.1]{Muzychuk2004}.  It does not seem feasible to always have all things get larger as we go from left to right.  The reason for this is that somehow one must communicate the power of $\tau_n$ together with the size of the blocks we are restricting to.  But as the power of $\tau_n$ gets larger, the size of the blocks we restrict to must get smaller.

We will show that the groups $\Pi({\bf k})$ are the Sylow $p$-subgroups of $5/2$-closed groups of degree $p^n$ that contain $\tau_n$.  We first give some examples of the groups $\Pi({\bf k})$, and a preliminary result.

\begin{example}\label{First example}
Let ${\bf k} = (0,0,\ldots,0)$.  Then $\tau_n^{p^{k_i}} = \tau_n$ for every $1\le i\le n$, and ${\cal B}_{n-0} = {\cal B}_n = \Z_{p^n}$, and so $\Pi({\bf k}) = \la\tau_n\ra$.
\end{example}

\begin{lem}\label{primary key quotient}
Let ${\bf k} = (k_1,k_2,\ldots,k_n)$, and define $\delta:{\cal B}_1\to\Z_{p^{n-1}}$ by $\delta(B_{1,j}) = j$.  Then $\delta(\Pi({\bf k})/{\cal B}_1) = \Pi(k_1,\ldots,k_{n-1})$.
\end{lem}

\begin{proof}
First observe that $\tau(B_{1,j}) = B_{1,j + 1}$.  Hence $\delta(\tau_n) = \tau_{n-1}$.  For $0\le i\le n-1$, let ${\cal B}_i'$ be the block system of $\la\tau_{n-1}\ra$ with blocks of size $p^i$.  As $\Pi({\bf k}) = \la P_{i,k_i}:1\le i\le n\ra$, it suffices to show that $\delta(P_{i,k_i}) = P_{i,k_i}'$, where $P_{i,k_i}' = \la \tau_{n-1}^{p^{i-1}}\vert_{B_{n-1-k_i}}:B_{n-1-k_i}\in{\cal B}_{n-1-k_i}'\ra$, and $1\le i\le n-1$.  Fix $1\le i\le n-1$.  As $P_{i,k_i} = \la \tau_n^{p^{i-1}}\vert_{B_{n-k_i}}:B_{n-k_i}\in{\cal B}_{n-k_i}\ra$, we have that $P_{i,k_i}/{\cal B}_1 =  \la (\tau_n^{p^{i-1}}/{\cal B}_1)\vert_{B_{n-k_i}/{\cal B}_1}:B_{n-k_i}\in{\cal B}_{n-k_i}\ra$.   As $(\tau_n^{p^{i-1}})^{p^{n - i}}(B_i) = \tau_n^{p^{n-1}}(B_i) = B_i$ and $\tau_n^{p^{i-1}}$ is a product of $p^{i-1}$ cycles of length $p^{n-i+1}$, we see that $\tau_n^{p^{i-1}}/{\cal B}_1$ is a product of $p^{i-1}$ cycles of length $p^{n-i}$.  As $\delta(\tau_n) = \tau_{n-1}$ is a $p^{n-1}$-cycle, the power of $\tau_{n-1}$ that gives $p^{i-1}$ cycles of length $p^{n-i}$ is $p^{i-1}$.  It is also clear that ${\cal B}_i/{\cal B}_1 = {\cal B}_{i-1}'$.  Hence $\delta({\cal B}_i) = {\cal B}_{i-1}'$.  Then $(\tau_n^{p^{i-1}}/{\cal B}_1)\vert_{B_{n-k_i}/{\cal B}_1} = \tau_{n-1}^{p^{i-1}}\vert_{B_{n-1-k_i}}$,
where $B_{n-1-k_i}\in{\cal B}_{n-1-k_i}'$, and  $\delta(P_{i,k_i}) = P_{i,k_i}'$.  The result follows.
\end{proof}

\begin{example}
Let ${\bf k} = (0,1,2,\ldots,n-2,n-1)$.  Then $\Pi({\bf k}) \cong\Z_p\wr_n\Z_p$, which is a Sylow $p$-subgroup of $S_{p^n}$ by \cite[Section 2]{Passman1968}.  We proceed by induction on $n$, with the first step of the induction being the primary key $(0)$, which gives $(\Z_p)_L$ as in Example \ref{First example}.  Inductively assume that the example is correct for the primary key $(0,1,2,\ldots,n-1)$ and consider the primary key ${\bf k} = (0,1,2,\ldots,n-1,n)$.  By Lemma \ref{primary key quotient}, we have that $P({\bf k})/{\cal B}_1$ is the subgroup corresponding to the primary key ${\bf k}' = (0,1,\ldots,n-1)$.  By the induction hypothesis, we see that $\Pi({\bf k})/{\cal B}_1\cong\Z_p\wr_{n-1}\Z_p$.  As $\fix_{\Pi({\bf k})}({\cal B}_1) = P_{n,n-1} = \la\tau_n^{p^{n-1}}\vert_{B_1}:B_1\in{\cal B}_1\ra$, we conclude that $\fix_{\Pi({\bf k})}({\cal B}_1)\cong \Z_p^{n-1}$.  Hence $\Pi({\bf k})\cong\Z_p\wr_n\Z_p$.
\end{example}

\begin{example}
Let $n = 3$ and ${\bf k} = (0,0,2)$.  It can be shown by inspecting $\Pi({\bf k})$ for all possible primary keys with at most three coordinates that ${\bf k}$ is the smallest primary key for which $\Pi({\bf k})$ is neither cyclic nor a wreath product.  By Example \ref{First example} and Lemma \ref{primary key quotient} we see that $\Pi({\bf k})/{\cal B}_1$ is cyclic.  Also, by definition, $\fix_{\Pi({\bf k})}({\cal B}_1) = \la\tau_3^{p^2}\vert_{B_2}:B_2\in{\cal B}_2\ra$ and so has order $p^p$.  Thus $\vert \Pi({\bf k})\vert = p^{p + 2}$.  Let $\Gamma$ be a Cayley digraph of order $p^3$ that contains $\Pi({\bf k})$.  Then, using the notation of \cite[Section 5.3]{Book}, we see that $\Gamma$ is a $(K,L)$-generalized wreath product with $K = \la p^2\ra$ and $L = \la p\ra$.  In general, for any primary key ${\bf k}$, if $\Pi({\bf k})$ is not cyclic, then any Cayley digraph of order $p^n$ that contains $\Pi({\bf k})$ will also be a generalized wreath product (although we will not give a proof of this fact here).
\end{example}

We now prove the main result of this section.

\begin{thrm}\label{graph}
Let $p\ge 3$ be prime, $n\ge 1$, and $G\le S_{p^n}$ that contains a regular cyclic subgroup and is $5/2$-closed. If $P$ is a Sylow $p$-subgroup of $G$, then $P$ is permutation isomorphic to $\Pi({\bf k})$ for some primary key ${\bf k}\in\Z_p^n$. Further, $P$ is $5/2$-closed.
\end{thrm}

\begin{proof}
By \cite[Corollary 1.1.10]{Book}, we see that $P$ is transitive.  As $P\le G$ and $G$ is $5/2$-closed, we have that $P^{(5/2)}\le G$. By Lemma \ref{5/2 closure of p-group}, we see $P^{(5/2)}$ is a $p$-group and so as $P$ is a Sylow $p$-subgroup of $G$, $P^{(5/2)} = P$.  It thus suffices to show that a $5/2$-closed $p$-subgroup of $S_{p^n}$ that contains a regular cyclic subgroup is permutation isomorphic to $\Pi({\bf k})$, for some primary key ${\bf k}\in \Z_p^n$.

We proceed by induction on $n$.  If $n = 1$, then $P = (\Z_p)_L$.  The result then follows with primary key $(0)$.  
Assume the result is true for $n - 1\ge 2$, and let $G\le S_{p^n}$ be transitive and $5/2$-closed.

We have observed that $P$ has ${\cal B}_i$, $0\le i\le n$ as block systems, and these block systems are the only block systems of $P$.  If $\vert P\vert = p^n$, then $P$ is a regular cyclic subgroup and the result follows by Example \ref{First example} with ${\bf k} = (0,0,\ldots,0)$. We thus assume without loss of generality that $\vert P\vert > p^n$.

Let ${\cal E}_i/{\cal B}_{i-1}$ be the ${\cal B}_i/{\cal B}_{i-1}$-fixer block system of $P$, $1\le i\le n-1$.  As $\vert P\vert > p^n$ and $p$ is odd, we have by Lemma \ref{bottom} that $\vert\fix_P({\cal B}_1)\vert > p^2$.  As $G$ is $5/2$-closed, we see $\fix_P({\cal B}_1) = \la\tau_n^{p^{n-1}}\vert_E:E\in{\cal E}_1\ra$.  As ${\cal E}_1 = {\cal B}_\ell$ for some $1\le \ell\le n-1$, we see $\fix_P({\cal B}_1) = P_{n,\ell}$.

We next use Theorem \ref{quotient 5/2-closed} to show $P/{\cal B}_1$ is $5/2$-closed.  Let $H\le P$ be transitive with ${\cal C}/{\cal B}_1$ a normal block system of $H/{\cal B}_1$ and ${\cal E}_{\cal C}/{\cal B}_1$ the ${\cal C}/{\cal B}_1$-fixer block system of $H/{\cal B}_1$.  By Lemma \ref{E F relationship} we have that ${\cal C}\preceq{\cal E}_1$ or ${\cal E}_1\prec{\cal E}_{\cal C}$.  Suppose ${\cal E}_{\cal C}\not\preceq{\cal E}_1$.  Then ${\cal  C}\prec{\cal E}_1$.  Let $E_C\in{\cal E}_{\cal C}$ and $C\in{\cal C}$ with $C\subseteq E_C$.  As ${\cal E}_{\cal C}\not\preceq{\cal E}_1$, $\WStab_{H/{\cal B}_1}(C/{\cal B}_1)^{C'/{\cal B}_1}\not = 1$ for some $C'\in{\cal C}$ with $C'\subseteq E_1$, where $E_1\in{\cal E}_1$ that contains $C$.  Let $L = \Stab_P(E_1)^{E_1}$.  As $H\le P$, $L$ contains $\Stab_H(E_1)^{E_1}$, and $\Stab_H(E_1)^{E_1}$ contains $\fix_H({\cal C})^{E_1}$ as $E_1$ is a union of blocks of ${\cal C}$ because ${\cal C}\prec{\cal E}_1$.  As $\WStab_{H/{\cal B}_1}(C/{\cal B}_1)^{C'/{\cal B}_1}\not = 1$ and $C,C'\subset E_1$, we see that $L$ is not a regular cyclic subgroup.  Let ${\cal B}_{E_1}$ be the set of blocks of ${\cal B}_1$ that are contained in $E_1$. By Lemma \ref{bottom} we see that $\fix_L({\cal B}_{E_1})$ has order at least $p^2$.  Then there exists $1\not = \rho\in P$ such that $\rho^{E_1}\in\fix_P({\cal B}_1)^{E_1}$ and $\rho^{E_1}\not\in\la\tau_n^{p^{n-1}}\ra^{E_1}$.  As $P$ is $5/2$-closed, we see that $\rho\vert_{E_1}\in P$, in which case the ${\cal B}_1$-fixer block system of $P$ is not ${\cal E}_1$, a contradiction.  This then implies that ${\cal E}_1\preceq {\cal E}_{\cal C}$.  As $H$ and ${\cal C}$ are arbitrary, by Theorem \ref{quotient 5/2-closed} we have that $G/{\cal B}_1$ is $5/2$-closed.  By the induction hypothesis, there exists a primary key ${\bf k}'\in\Z_p^{n-1}$ such that $P/{\cal B}_1 = \Pi({{\bf k}'})$.

Let ${\bf k'} = (k_1',k_2',\ldots,k_{n-1}')$, and ${\bf k} = (k_1',k_2',\ldots,k_{n-1}',\ell) = (k_1,k_2,\ldots,k_n)$.  We now show that ${\bf k}$ is a primary key.  Note that ${\cal E}_1 = P_{n,\ell}$, and by induction we have that ${\cal E}_2/{\cal E}_1 = P_{n-1,k_{n-1}}/{\cal B}_1$.  By Lemma \ref{p^kfixers} we see that ${\cal E}_1\preceq{\cal E}_2$.  Hence $k_{n-1}\le k_n$.  As $1\le \ell\le n-1$, $k_n\le n-1$ and so ${\bf k}$ is a primary key.

We now show that $\Pi({\bf k}) = P$.  By the induction hypothesis and Lemma \ref{primary key quotient} we see that $P/{\cal B}_1 = \Pi({\bf k})/{\cal B}_1$.  As $P$ is $5/2$-closed, we see that $\tau_n^{p^{k_i}}\vert_{B_{n-i}}\in P$ for every coordinate $k_i$ of ${\bf k}$ and $B_{n-i}\in{\cal B}_{n-i}$.  Hence $P_{i,k_i}\le P$ for each coordinate $k_i$ of ${\bf k}$ and so $\Pi({\bf k})\le P$.  Let $\gamma\in P$.  Then there exists $\delta\in \Pi({\bf k})$ such that $\gamma/{\cal B}_1 = \delta/{\cal B}_1$.  Hence $\delta^{-1}\gamma\in\fix_P({\cal B}_1) = P_{n,\ell}\le \Pi({\bf k})$.  Thus $\gamma\in \Pi({\bf k})$ and $P = \Pi({\bf k})$.  The result follows by induction.
\end{proof}

We will show in \cite{Dobson2023epreprint} that if ${\bf k}$ is a primary key such that $\Pi({\bf k})$ cannot be written as a nontrivial wreath product, then every minimal transitive subgroup of $\Pi({\bf k})$ is a regular cyclic subgroup.  By \cite[Theorem 33]{Dobson2005} this will imply that any permutation group $G$ of $S_{p^n}$ with Sylow $p$-subgroup $\Pi({\bf k})$ is either $2$-transitive or $\Pi({\bf k})\tl G$.  The normalizers of such $\Pi({\bf k})$ are then calculated in \cite{Dobson2023epreprint}.  This results not only in a classification of $5/2$-closed groups of odd prime power degree which contain a regular cyclic subgroup, but also a classification of $2$-closed groups of odd prime power degree which contain a regular cyclic subgroup.  These, in turn, give the automorphism group of every circulant digraph of odd prime power degree, as well as the automorphism group of every configuration of odd prime power order that contains a regular cyclic subgroup.

\providecommand{\bysame}{\leavevmode\hbox to3em{\hrulefill}\thinspace}
\providecommand{\MR}{\relax\ifhmode\unskip\space\fi MR }
\providecommand{\MRhref}[2]{%
  \href{http://www.ams.org/mathscinet-getitem?mr=#1}{#2}
}
\providecommand{\href}[2]{#2}

\end{document}